\documentclass[reqno,centertags,12pt]{amsart}
\usepackage{amsmath,amsthm,amscd,amssymb,latexsym,verbatim}
\usepackage{amssymb}

\usepackage{graphicx,epsf,cite}

\usepackage[shortlabels]{enumitem}

\usepackage{xcolor}

-\marginparwidth \oddsidemargin -4mm \evensidemargin \oddsidemargin

\newtheorem{theorem}{Theorem}[section]

\newtheorem{proposition}[theorem]{Proposition}
\newtheorem{lemma}[theorem]{Lemma}
\newtheorem{corollary}[theorem]{Corollary}
\theoremstyle{definition}
\newtheorem{definition}[theorem]{Definition}

\newcounter{smalllist}

\allowdisplaybreaks
\numberwithin{equation}{section}

\newcommand{\lb}{\label}

\newcommand{\supp}{\text{\rm{supp}}}

\newcommand{\beq}{\begin{equation}}
\newcommand{\eeq}{\end{equation}}

\newcommand{\bal}{\begin{align}}
\newcommand{\eal}{\end{align}}
\newcommand{\bals}{\begin{align*}}
\newcommand{\eals}{\end{align*}}

\newcommand{\bbN}{{\mathbb{N}}}
\newcommand{\bbR}{{\mathbb{R}}}

\newcommand{\bbP}{{\mathbb{P}}}
\newcommand{\bbE}{{\mathbb{E}}}
\newcommand{\bbZ}{{\mathbb{Z}}}

\newcommand{\bbS}{{\mathbb{S}}}

\newcommand{\calE}{{\mathcal E}}

\newcommand{\calF}{{\mathcal F}}
\newcommand{\calH}{{\mathcal H}}

\newcommand{\eps}{\varepsilon}

\newcommand{\al}{\alpha}
\newcommand{\be}{\beta}

\begin{document}
\title[Quantitative Homogenization for Combustion in Random Media]
{Quantitative Homogenization for \\ Combustion in Random Media}

\author{Yuming Paul Zhang and Andrej Zlato\v s}

\address{\noindent Department of Mathematics \\ University of
California San Diego \\ La Jolla, CA 92093 \newline Email: \tt
zlatos@ucsd.edu}

\address{\noindent Department of Mathematics \\ University of
California San Diego \\ La Jolla, CA 92093 \newline Email: \tt
yzhangpaul@ucsd.edu}


\begin{abstract} 
We obtain the first quantitative stochastic homogenization result for reaction-diffusion equations, for ignition reactions in dimensions $d\le 3$ that either have finite ranges of dependence or are close enough to such reactions, and for solutions with initial data that approximate characteristic functions of general convex sets.  We show algebraic rate of convergence of these solutions to their homogenized limits, which are (discontinuous) viscosity solutions of certain related Hamilton-Jacobi equations.
\end{abstract}

\maketitle

\section{Introduction} \lb{S1}

A basic model of combustion processes in random media is the reaction-diffusion equation
\beq\lb{1.1}
u_t=\Delta u+f(x,u,\omega)
\eeq
with  $(t,x)\in (0,\infty)\times \bbR^d$ and
 $\omega$ an element of some probability space $(\Omega,\calF,\bbP)$.
Its solutions $u$ represent normalized temperature of the combusting medium, taking values between 0 and 1, and the
reaction function $f$ is of the {\it ignition type}, satisfying $f(\cdot, u,\cdot) \equiv 0$ whenever $u\in [0, \theta_1]\cup\{1\}$, for some $\theta_1\in(0,1)$. 

This model, with homogeneous reactions $f(x,u,\omega)=f(u)$ goes back to pioneering works by Kolmogorov, Petrovskii, and Piskunov \cite{KPP}, and Fisher \cite{Fisher}. In this case it is well known that solutions to \eqref{1.1} propagate ballistically in all directions at a constant speed $c^*$ in the sense that a solution with initial data  close to the characteristic function of some (not too small) set $A\subseteq \bbR^d$ is in a sense close to the characteristic function of the set $A+c^*t B_1(0)$ at any large time $t>0$.  We refer to \cite{aronson1975,aronson1978} for various results in the homogeneous reaction case, and to the reviews  \cite{47souganidis,xin2000, Berrev} for other related developments and references.

%

The setting of heterogeneous reactions is much more complicated as one cannot expect the same propagation speed in all directions --- or indeed any propagation speed at all --- for general $f$.  However, when an environment is random, and sufficiently so (e.g., when $f$ is i.i.d.~in space in some sense or, more generally, stationary ergodic), large space-time  scale dynamics of physical processes occurring inside it frequently appear  as if the environment were homogeneous (albeit non-isotropic).  This phenomenon, called {\it homogenization}, is a result of large-scale averaging of the random heterogeneities in the medium, and in the setting of \eqref{1.1} would also mean existence of direction-dependent asymptotic propagation speeds of solutions.

While existence of homogenization has long been known in various settings, in particular for (first-order as well as ``viscous'' second-order) Hamilton-Jacobi equations (the literature is vast; the reader can consult \cite{LioSou,armstrong2015,LinZla,LPV,LioSouPer, sou1999,RezTar,kosy} and references therein), until recently it has been proved for reaction-diffusion equations only in one spatial dimension $d=1$, even in the simplest heterogeneous setting of spatially periodic reactions $f$.  The main reason for this is that in the case of reaction-diffusion equations, the (homogenized) large-space-time limits of  solutions to \eqref{1.1} are in fact expected to be (discontinuous) characteristic functions of time-expanding regions, which are also (viscosity) solutions to a very different PDE, the (first-order) Hamilton-Jacobi equation \eqref{1.5} below with some $f$-dependent ``speed'' $c^*:\bbS^{d-1}\to(0,\infty)$.  When this fact is coupled with complications caused by potentially very non-trivial geometries of the boundaries of these regions in dimensions $d\ge 2$, it is not surprising that the question of homogenization in this setting becomes substantially more challenging.

In fact, the first proofs of stochastic homogenization for \eqref{1.1} in dimensions $d\ge 2$ have only been provided recently and only for ignition reactions (we also note that a homogenization result for {\it KPP reactions}, satisfying $f(\cdot,0,\cdot)\equiv f(\cdot,1,\cdot)\equiv 0$, and $0<f(x,u,\omega)\le f_u(x,0,\omega)u$ when $u\in(0,1)$, was stated without proof in the paper \cite{LioSou} by Lions and Souganidis).  First, Lin and the second author obtained a number of conditional homogenization results for general reactions, and showed that the hypotheses in these apply, in particular, to isotropic stationary ergodic ignition reactions in dimensions $d\le 3$ \cite{LinZla}.   We then showed that homogenization also holds for general stationary ignition reactions in dimensions $d\le 3$ that either have finite ranges of dependence (which is a continuous version of an i.i.d.~environment) or are in some sense close to such reactions \cite{ZhaZla} (we refer to \cite{ZlaInhomog, LinZla, ZhaZla} for further discussion on this, including the reason for the not-just-technical and physically relevant limitation to $d\le 3$, which we also briefly mention after Definition \ref{D.2} below).  The hypotheses {\bf (H1)--(H4)} below in fact mirror those from  \cite{ZhaZla}, although for the sake of simplicity we will not consider  here the most general form of the hypotheses in \cite{ZhaZla}.

We also note that when it comes to {\it periodic}  reactions (i.e., $f(x,u,\omega)=f(x,u)$ and periodic in $x$), homogenization was proved for {\it monostable} ones (as KPP but without requiring $f(x,u)\le f_u(x,0)u$, so  KPP reactions are included) by Alfaro and Giletti \cite{AlfGil} for initial data with smooth convex supports.  This was extended to general convex supports in \cite{LinZla}, where  homogenization was also proved for periodic ignition reactions and quite general initial data in any dimension. We also refer to the work \cite{MajSou} by Majda and Souganidis for the case of \eqref{1.1} with homogeneous KPP reactions and periodic first-order advection terms.

Given how recent the above results are, it is no surprise that until now no {\it quantitative estimates} on the speed of convergence of solutions to \eqref{1.1} to their homogenized limits have been obtained.  The goal of this paper is to address this question for the random ignition reactions considered in \cite{ZhaZla} (see Theorem \ref{T.1.1} below).  This involves the study of the large-space-time-scale version of \eqref{1.1}, that is,
\beq\lb{1.4}
( u_\eps)_t=\eps\Delta u_\eps+{\eps}^{-1}f\left({\eps}^{-1}{x}, u_\eps,\omega\right)
\eeq
with a small $\eps>0$, so that solutions $u$ to \eqref{1.1} give rise to those for \eqref{1.4} via
\beq\lb{1.3}
u_\eps(t,x,\omega):=u\left({\eps}^{-1}{t},{\eps}^{-1}{x},\omega\right).
\eeq
Our main result in \cite{ZhaZla} is  that if initial data for \eqref{1.4} sufficiently well approximate the characteristic function of some open set $A\subseteq\bbR^d$ as $\eps\to 0$, then the solutions $u_\eps$ almost surely converge  to the characteristic function of a set $\Theta^{A,c^*}\subseteq (0,\infty)\times\bbR^d$, in the sense of locally uniform convergence on the complement of $\partial \Theta^{A,c^*}$ (i.e., where this characteristic function is continuous).  In fact, as is shown in \cite{LinZla}, $\chi_{\Theta^{A,c^*}}$ is a viscosity solution with initial data $\chi_A$ to the 
 deterministic homogeneous (non-isotropic) Hamilton-Jacobi equation 
 \beq\lb{1.5}
\bar{u}_t=c^*\left(-\frac{\nabla\bar{u}}{|\nabla \bar{u}|}\right)|\nabla \bar{u}|,
\eeq
where $c^*(e)>0$ is a (deterministic asymptotic) {\it front speed} for \eqref{1.1} in direction $e\in\bbS^{d-1}$, which  exists for each $e$ and the function $c^*:\bbS^{d-1}\to(0,\infty)$ is Lipschitz \cite{ZhaZla}.  

One can therefore view \eqref{1.5} as the homogenization limit of \eqref{1.4}.  We then show in Theorem \ref{T.1.1} below  that when the initial set $A$ is bounded and convex, then convergence to this limit is algebraic in $\eps$ (with some power $\sigma>0$), with a probability that exponentially converges to 1 as $\eps\to 0$.  Specifically, we refer here to convergence of the $\theta$-super-level set
\[
\Gamma_{u_\eps,\theta}(t,\omega):=\left\{x\in\bbR^d\,|\,u_\eps(t,x,\omega)\geq\theta\right\}
\]
 of  $u_\eps(\cdot,t,\omega)$ to $\Theta^{A,c^*}(t):=\{x\in\bbR^d\,|\, (t,x)\in \Theta^{A,c^*}\}$, for each fixed $\theta\in(0,1)$ and uniformly on bounded time intervals. 
  We also note that in this convex $A$ case, the set
$\Theta^{A, c^*}(t)$ is also convex and was in fact shown in \cite[Theorem 1.4(iii)]{LinZla} to have the relatively simple form 
\beq\lb{6.2}
\Theta^{A, c^*}(t)= \bigcap_{e\in\bbS^{d-1}} \left\{ x\in \bbR^d\,\bigg|\, x\cdot e < \sup_{y\in A} y\cdot e+c^*(e)t\,\right\}.
\eeq

Theorem \ref{T.1.1} is hence a {\it quantitative stochastic homogenization} result for \eqref{1.1}, which is to the best of our knowledge the first one for reaction-diffusion equations.
The basis of our analysis are results from our paper \cite{ZhaZla}, primarily those in Proposition \ref{P.3.1} below.  These are quantitative estimates on the fluctuations of arrival times at any point in $\bbR^d$ of special solutions to \eqref{1.1} with half-space-like initial data, and were obtained via a method inspired by related pioneering results of Armstrong and Cardaliaguet \cite{armstrong2015} for Hamilton-Jacobi equations with non-convex finite-range-of-dependence Hamiltonians.  We note that in the case of Hamilton-Jacobi homogenization, the limiting PDE is again a Hamilton-Jacobi equation; this differs from our reaction-diffusion case, where the homogenization limit of \eqref{1.4} is  \eqref{1.5} (see  \cite{LinZla} for further discussion concerning this relationship).  


We note that while we could prove our results in more generality, in particular, include in Theorem \ref{T.1.1} also reactions that are less well approximated by those with finite ranges of dependence (see in particular hypothesis {\bf (H4')} and Example 1.6 in \cite{ZhaZla}), we chose not to do so here for the sake of clarity.

Let us now move to the precise statements of our hypotheses, which are from \cite{ZhaZla}, and to our main result.  We start with the definition of stationary reactions.

\begin{definition}\lb{D.1}
Let $(\Omega,\calF,\bbP)$ be a probability space that is endowed with a group of measure-preserving bijections $\{{\Upsilon_y:\Omega\to\Omega}\}_{y\in\bbR^d}$ such that for all $y,z\in \bbR^d$,
\[
\Upsilon_y\circ\Upsilon_z=\Upsilon_{y+z}.
\]
A reaction function $f:\bbR^d\times [0,1]\times\Omega\to [0,\infty)$, with the random variables $X_{x,u} :=f(x,u,\cdot)$ being $\calF$-measurable for all $(x,u)\in \bbR^d\times [0,1]$, is called \textit{stationary} if for each $(x,y,u,\omega)\in \bbR^{2d}\times [0,1]\times\Omega$ we have
\[
f(x,u,\Upsilon_y\omega) = f(x+y,u,\omega).
\] 
The \textit{range of dependence} of such  $f$ is the infimum of all $r\in\bbR^+\cup\{\infty\}$ such that 
\[
\calE(U)\text{ and }\calE(V) \text{ are $\bbP$-independent}
\]
for any $U,V\subseteq \bbR^d$ with $d(U,V)\geq r$,
where $\calE(U)$ is the $\sigma$-algebra generated by 
the family of random variables 
$\{X_{x,u}  \,|\, (x,u)\in U\times [0,1]\}$
and $d(\cdot,\cdot)$ is the standard distance in $\bbR^d$.
\end{definition}


Since we are interested in ignition reactions, we assume the following hypothesis. 

\smallskip
\begin{itemize}
    \item[\textbf{(H1)}] The reaction  $f$ is stationary, Lipschitz in both $x$ and $u$ with constant ${M}\geq 1$, 
   and there are $\theta_1\in (0,\frac 12)$, $m_1>1$, and $\al_1>0$ such that
    $f(\cdot,u,\cdot)\equiv 0$ for $u\in [0,\theta_1]\cup\{1\}$, $f(\cdot,u,\cdot)\ge \al_1 (1-u)^{m_1}$ for $u\in [1-\theta_1,1)$, and  $f$ is non-increasing in $u\in [1-\theta_1,1)$.
\end{itemize}
\smallskip

In fact, we need to assume slightly more, since  one cannot hope for general reactions satisfying \textbf{(H1)} to lead to homogenization for  \eqref{1.1} as described above, even for homogeneous reactions $f(x,u,\omega)=f(u)$.  Indeed, if $f$ is allowed to vanish at some intermediate value $\theta'\in(\theta_1,1-\theta_1)$ and is also ``sufficiently larger'' on  $(\theta_1,\theta')$ than on $(\theta',1)$, solutions typically form ``plateaus'' at value $\theta'$ (or another intermediate value) whose widths grow linearly in time, and so these plateaus will not disappear as $\eps\to 0$ and the scaling \eqref{1.3} is applied   (see \cite{ZlaInhomog, ZlaBist} for more details).
To avoid this scenario, we make the following definition.

\begin{definition}\lb{D.2}
A reaction $f$ satisfying  \textbf{(H1)} is a stationary {\it pure ignition} reaction if for each $\eta>0$ we have
\[
\inf_{\substack{(x,\omega)\in\bbR^d\times\Omega \\ \theta_{x,\omega} + \eta < 1-\theta_1 }} f(x,\theta_{x,\omega}+\eta,\omega)>0,
\]
where the {\it ignition temperature} $\theta_{x,\omega}$ is defined by
\[
\theta_{x,\omega} := \sup \{\theta\ge 0 \,|\, f(x,u,\omega)=0 \text{ for all $u\in[0,\theta]$}\} \qquad (\in [\theta_1,1-\theta_1)).
\]
\end{definition}

As the second author showed in \cite{ZlaInhomog}, the linearly growing plateaus scenario may occur even for pure ignition reactions, but only in dimensions $d\ge 4$ (this relates to transience of Brownian motion in $\bbR^{d-1}$).  Therefore our main hypothesis on the reaction $f$ is the following.


    

  
  





 \smallskip

\begin{itemize}
    \item[\textbf{(H2)}]  $f$ is a stationary pure ignition  reaction and $d\leq 3$.
\end{itemize}
\smallskip

Finally, we will assume that $f$ either has a finite range of dependence, or is close enough to such reactions and has certain uniform decay in $u$ near $u=1$.  
The following two hypotheses relate to the second alternative.

\smallskip

\begin{itemize}
\item[\textbf{(H3)}] There are $m_3\geq 1$ and $\al_3>0$ such that for all $\eta\in (0,\frac 12\theta_1]$ we have 
\[
\inf_{\substack{
(x,\omega)\in\bbR^d\times\Omega\\
    u\in [1- \theta_1/2,1]}}\left(f(x,u-\eta,\omega)-f(x,u,\omega)\right)\geq \al_3 \eta^{m_3}.
    \]
\end{itemize}

\smallskip

\begin{itemize}
\item[\textbf{(H4)}] 
There are $m_4,n_4,\al_4>0$ such that for each $n\geq n_4$, there exists a stationary reaction $f_n$
with range of dependence $\le n$ and $\|f_n-f\|_\infty \leq \al_4 {n^{-m_4}}$.
\end{itemize}

\smallskip

We are now ready to state our  main result. In it  we denote $B_r(A):=A+(B_r(0)\cup\{0\})$ and $A^0_r:=A\backslash\overline{B_r(\partial A)}$ for $A\subseteq\bbR^d$ and $r\ge 0$ (in particular, $A^0_0$ is the interior of $A$).  Note that if $A$ is convex, so are $B_r(A)$ and $A^0_r$.  We also let $\tilde \sigma:=\min\left\{\frac{1}{8m_1},\frac{m_4}{4m_3+8m_4}\right\}$, where we ignore the second term when $f$ is assumed to have a finite range of dependence (and so {\bf (H3)--(H4)} is not assumed).

\begin{theorem}\lb{T.1.1}
Assume that $f$ satisfying \textbf{(H2)} either has a finite range of dependence or satisfies \textbf{(H3)}--\textbf{(H4)}.  There is a Lipschitz function $c^*:\mathbb{S}^{d-1}\to (0,\infty)$ such that if
 $u_\eps$ solves \eqref{1.4} and for some open bounded convex set $A \subseteq \bbR^d$ and some $\nu> 0$ we have
 \[
(1-\theta_1)\chi_{A^0_{ \eps^\nu}}\leq {u_\eps}(0,\cdot,\omega)\leq \chi_{B_{\eps^\nu}(A)}
\]
for each $\eps>0$,
then  the following holds with $\sigma:=\frac 12\min\left\{\tilde\sigma,\nu \right\}$.
For any $\theta\in (0,1)$ and $T_0>0$, there are constants  $C_0=C_0({M},\theta_1,m_1,\alpha_1,A,\theta)$ and $\eps_0>0$
such for all $\eps\in (0,\eps_0]$ we have
\[
\bbP\left[ \left(\Theta^{A,c^*}(t) \right)^0_{\eps^{{\sigma}}} \subseteq \Gamma_{u_\eps,\theta}(t,\cdot)\subseteq B_{\eps^{{\sigma}}} \left(\Theta^{A,c^*}(t) \right)  \text{ for all }t\in[C_0\eps ,T_0]\right]\geq  1- \exp\left(-\eps^{-{2\sigma}}\right).
\]

\end{theorem}

\noindent {\it Remarks. } 1.  
The limitation of the above estimate to times $t\ge  C_0\eps$ is necessary because if $\theta$ is close to 1, it takes time $O(\eps)$ for $u_\eps$ to reach the value $\theta$.  If $\theta < 1-\theta_1$, then it is not difficult to show that Theorem \ref{T.1.1} extends to include  $t\in [0,T_0]$ in the statement because both inclusions then hold for all $(t,\omega)\in[0,C_0\eps]\times\Omega$ when $\eps>0$ is small enough.
\smallskip

2. 
We can also determine on which parameters $\eps_0$ depends.
It turns out that there is some $\eta_*=\eta_*(M, \theta_1,m_1,\alpha_1)>0$ such that if for some $\xi>0$ we have
\beq\lb{1.111}
\inf_{\substack{(x,\omega)\in\bbR^d\times\Omega \\ u\in[\theta_{x,\omega} + \eta_* ,1-\theta_1] }} f(x,u,\omega)\ge \xi,
\eeq
then $\eps_0$   can be chosen to depend only on $A,\nu,\theta, T_0$ plus on
\beq\lb{const}
{M},\theta_1,m_1,\alpha_1, \xi,  \text{\,\,\,\,and either\,\,\,\,} \rho \text{\,\,\,\,or\,\,\,\,} m_3,\al_3,m_4,n_4,\al_4,
\eeq
depending on whether we assume \textbf{(H2)} plus $f$ having range of dependence at most $\rho\in[1,\infty)$, or we assume \textbf{(H2)--(H4)}.  See the proof of Theorem \ref{T.1.1} for details on this.
\smallskip

3.
Lemma \ref{L.2.1} below shows that if $A$  is unbounded (but still convex), then Theorem \ref{T.1.1} holds locally uniformly, that is, with $\left(\Theta^{A,c^*}(t) \right)^0_{\eps^{{\sigma}}}$ and $\Gamma_{u_\eps,\theta}(t,\cdot)$ replaced by their intersections with $B_N(0)$, for any $N\in\mathbb N$ ($C_0$ and $\eps_0$ then also depend on $N$).
\smallskip

4.  We make here no attempt to optimize the power $\sigma$
in Theorem \ref{T.1.1}.

\smallskip




\smallskip

\subsection{Organization of the Paper and Acknowledgements}
In Section \ref{S2} we collect several important preliminary results as well as most of the notation used later.  In Section \ref{S3}, we construct certain regularized approximations of the sets $\Theta^{A,c^*}$, which are then used in the proof of Theorem~\ref{T.1.1} in Section \ref{S4}.

YPZ acknowledges partial support by an AMS-Simons Travel Grant.  AZ acknowledges partial support by  NSF grant DMS-1900943 and by a Simons Fellowship.

\section{Preliminaries and Notation}\lb{S2}

Most of the results in this section are from \cite{ZhaZla}, and we reproduce them here for the reader's convenience. Many of them hold uniformly in $\omega$ and even without assuming stationarity of the reaction, and in these we will therefore replace {\bf (H1)} by the following weaker hypothesis.\smallskip

\begin{itemize}
\item[\textbf{(H1')}] 
$f$ satisfies \textbf{(H1)} except possibly the stationarity hypothesis.
\end{itemize}
\smallskip

We collect the needed results assuming {\bf (H1')} in the following subsection.


\subsection{General Ignition Reactions}

 Let us start with a basic lower bound which shows that general solutions to \eqref{1.1} propagate with speed no less than some $c_0>0$  (see \cite{ZlaInhomog}).  Consider  the  largest $M$-Lipschitz function $F_0:[0,1]\to[0,\infty)$ such that $F_0(u)\le \alpha_1(1-u)^{m_1} \chi_{[1-\theta_1,1]}(u)$ for all $u\in[0,1]$, which of course guarantees that $f(x,\cdot,\omega)\ge F_0$ for all $\omega\in\Omega$ when $f$ satisfies {\bf (H1')}. 
 Then $F_0$ is a homogeneous pure ignition reaction, and we let $c_0>0$  be its traveling front speed (i.e., such that the PDE $u_t=u_{xx} + F_0(u)$ in one spatial dimension has a {\it traveling front} solution $u(t,x)=U(x-c_0t)$, with $U(-\infty)=1$ and $U(\infty)=0$).  

\begin{lemma}\lb{L.2.2}
There exists $\theta_2=\theta_2({M},\theta_1,m_1,\alpha_1)<1$ 
such that for each $c<c_0$ and $\theta<1$, there is $\kappa_0=\kappa_0({M},\theta_1,m_1,\alpha_1,c,\theta)\geq 1$ such that the following holds. If $u:(0,\infty)\times \bbR^d\to [0,1]$ is a solution to \eqref{1.1} with $f$ satisfying \textbf{(H1')} 
and with some $\omega\in\Omega$, 
and if $u(t_0,y)\geq \theta_2$ for some $t_0\ge 1$ and $y\in\bbR^d$, then for all $t\geq t_0+\kappa_0$, 
\[
\inf_{|x-y|\leq c(t-t_0) }u(t,x) \geq \theta.
\]
If also $u_t\ge 0$, then this clearly holds with any $t_0\ge 0$ (and $\kappa_0$ increased by 1).
\end{lemma}


Let
\beq\lb{2.10}
\theta^*:=\frac{\min\{1-\theta_2,\,\theta_1\}}4.
\eeq
The next few results are from \cite{ZhaZla}, and stated there with $\theta_2=\theta_2(M,\frac 12\theta_1,m_1,\alpha_1(1-\frac 18\theta_1)^{m_1-1})$ in the definition of $\theta^*$; however, the remark after \cite[Lemma 2.1]{ZhaZla} explains that they also hold with \eqref{2.10} and $\theta_2=\theta_2(M,\theta_1,m_1,\alpha_1)$ (moreover, this distinction will be of no consequence here).
The first of these is \cite[Lemma 2.8]{ZhaZla},
which provides an upper bound on $\kappa_0({M},\theta_1,m_1,\alpha_1,\frac{c_0}4,\theta)$ from Lemma \ref{L.2.2} as $\theta\to 1$.


\begin{lemma}\label{L.2.8}
Let $u:[0,\infty)\times\bbR^d\to [0,1]$ solve \eqref{1.1} with $f$ satisfying \textbf{(H1')} and some $\omega\in\Omega$.  There is $D_1=D_1({M},\theta_1,m_1,\alpha_1)$ such that if $u(t_0,y)\geq 1-\theta^*$ for some $t_0\ge 1$ and $y\in\bbR^d$, then for any  $\theta\in[1-\theta^*,1)$ and $t\geq t_0 + D_1 (1-\theta)^{1-m_1}$  we have
\[
\inf_{|x-y|\le c_0(t-t_0)/4} u(t,x)\geq \theta .
\]
\end{lemma}

Throughout the rest of the paper we will primarily use Lemma \ref{L.2.2} with $c=\frac{c_0}{2}$ and $\theta=1-\theta^*$, 
hence we define
\[
\kappa_0:=\kappa_0 \left({M},\theta_1,m_1,\alpha_1, \frac {c_0}2,1-\theta^* \right).
\]


The next result is \cite[Lemma 2.2]{ZhaZla}, which constructs smooth initial data $u_{0,S}$ that approximate $(1-\theta^*)\chi_{S}$ and the corresponding solutions satisfy $u_t\ge 0$.

\begin{lemma}\lb{L.2.4}
There is $R_0=R_0({M},\theta_1,m_1,\alpha_1)\ge 1$ such that for any $f$ satisfying \textbf{(H1')}  and $S\subseteq \bbR^d$, there is a smooth function $u_{0,S}$ satisfying 
\[
\Delta u_{0,S}+{F_0}(u_{0,S})\geq 0,
\]
and
\[
(1-\theta^*)\chi_{S}\leq u_{0,S}\leq (1-\theta^*)\chi_{B_{R_0}(S)}.
\]

\end{lemma}



\smallskip

The following counterpart to Lemma \ref{L.2.2} (see \cite[Lemma 2.2]{LinZla} and \cite[Lemma 2.5]{ZhaZla}) yields an upper bound on the speed of propagation of perturbations of solutions to \eqref{1.1}.

\begin{lemma}\lb{L.2.1}
Let $u_1,u_2:[0,\infty)\times\bbR^d\to [0,1]$ be, respectively, a subsolution and a supersolution 
to \eqref{1.1} with some $f$ satisfying \textbf{(H1')} and some $\omega\in\Omega$,  and let $r>0$ and $y\in\bbR^d$.  If $u_1(0,\cdot)\leq u_2(0,\cdot)$ on $B_r(y)$, then for all $(t,x)\in[0,\infty)\times \bbR^d$ we have
\[
u_1(t,x)\leq u_2(t,x)+2d\, e^{ \sqrt{M/{d}\,}\left(|x-y|-r+2 \sqrt{Md\,}\,t\right)}.
\]
\end{lemma}

This estimate yields the following two results.  The first of them is just \cite[Corollary 2.6]{ZhaZla}, and  in the second we let
\[
T_u(x):= \inf\{t\ge 0\,|\, u(t,x)\ge 1-\theta^*\}.
\]


\begin{corollary}\lb{C.2.1}
If $u:[0,\infty)\times\bbR^d\to [0,1]$ solves \eqref{1.1} with some $f$ satisfying \textbf{(H1')} and some $\omega\in\Omega$, then for any $t\geq 0$
we have 
\[
\{x\in\bbR^d\,|\, u(t,x)\geq 1-\theta_1\}\subseteq  B_{c_1t+\kappa_1} \big( \{ x\in\bbR^d\,|\, u(0,x)\geq \theta_1\} \big),
 \]
where
\[
c_1:=2\sqrt{Md\,}>c_0\qquad \text{ and } \qquad\kappa_1:=1+ \sqrt{d/{M}\,} \ln \frac{2d}{1-2\theta_1}.
\]
\end{corollary}


\begin{corollary}\label{L.2.9}
Let $u_1,u_2:[0,\infty)\times\bbR^d\to [0,1]$ solve \eqref{1.1} with $f$ satisfying \textbf{(H1')} and some $\omega\in\Omega$. There is $D_2=D_2({M},\theta_1,m_1,\alpha_1)\ge 1$ such that if $u_1(0,\cdot)\le u_2(t_0,\cdot)$ on $B_R(0)$
for some $t_0\geq 0$ and
$R\ge D_2(1+{T}_{u_1}(0)),$
then
\[
{T}_{u_1}(0)\geq  T_{u_2}(0)-t_0-\kappa_0.
\]
\end{corollary}

\begin{proof}
By Lemma~\ref{L.2.1}, we have
\[
u_1(t,0)\leq u_2(t+t_0,0)+2d\,e^{2M t -  \sqrt{M/d\,}R}
\]
 for all 
$t\ge 0$.  Hence,
\[
u_2(T_{u_1}(0)+t_0,0) \ge u_1(T_{u_1}(0),0)  - \theta^* \ge  1-2\theta^*
\]
as long as 
\[
R \geq  2\sqrt{Md\,}{T}_{u_1}(0)+\sqrt{d/M}\ln\frac {2d}{\theta^*},
\]
which will be guaranteed if we let $D_2:=2\sqrt{Md}\ln\frac{2d}{\theta^*}$.
But then Lemma \ref{L.2.2} yields
\[
u_2(T_{u_1}(0)+t_0+\kappa_0,0)\geq 1-\theta^*
\]
and the result follows.
\end{proof}

\subsection{Stationary Ignition Reactions}

Identification of the front speeds $c^*(e)$ for \eqref{1.1} with a stationary reaction $f$ is based on the analysis of the dynamics of special solutions starting from approximate characteristic functions of half-spaces. Specifically, for any $e\in \bbS^{d-1}$ let
$\calH_e^-:=\{x\in\bbR^d\, | \,x\cdot e\le  0\}$, and for any $y\in\bbR^d$  let $u=u(t,x,\omega;e,y)$ be the solution to
\begin{equation} \lb{4.1}
\begin{aligned}
& u_t=\Delta u+f(x,u,\omega)\qquad && \text{ on }(0,\infty)\times \bbR^d,\\
& u(0,\cdot,\omega;e,y)=u_{0,\calH_e^-+y}\qquad && \text{ on } \bbR^d,
\end{aligned}
\end{equation} 
where $u_{0,\calH_e^-+y}$ satisfies Lemma~\ref{L.2.4} with $S:=\calH_e^-+y$. 
Then for any $(x,\omega)\in \bbR^d\times\Omega$ let
\[
T(x,\omega;e,y):=\inf\{t\geq 0\,|\, u(t,x,\omega;e,y)\geq 1-\theta^*\},
\]
which one can think of as the arrival time of the solution from \eqref{4.1} at $x$.
Corollary 2.7 and  Propositions 3.8, 4.2, and 5.1 in \cite{ZhaZla} (see also (5.5) in \cite{ZhaZla}) now yield the following fluctuation estimate for $y=0$, which immediately extends to all $y\in\bbR^d$ by stationarity of $f$.



\begin{proposition}\lb{P.3.1}
Let $f$  satisfying \textbf{(H2)} either have range of dependence at most $\rho\in[1,\infty)$ or satisfy \textbf{(H3)--(H4)}.   
Then there is $\bar C\ge 1$  such that if in the former case we let
\beq\lb{b.1'}
\beta:=1-\frac{1}{2m_1}, 
\eeq
and in the latter case we let
\beq\lb{b.1''}
\beta:=1-\min\left\{\frac{1}{2m_1},\frac{m_4}{m_3+2m_4}\right\},
\eeq
then  for each $e\in\bbS^{d-1}$, $\lambda\ge 0$, and $x,y\in\bbR^d$ with $(x-y)\cdot e\geq 1$ we have
\[
\bbP\left[ \big| T(x,\cdot\,;e,y)-\bbE [T(x,\cdot\,;e,y)] \big| \geq \lambda\right] \leq 2\exp \left( -\bar C^{-2}\lambda^2((x-y)\cdot e)^{-2\beta} \right).
\]
Moreover, there is $ \bar{T}(e)\in [\frac 1{c_1},\frac 1{c_0}]$ (depending on $f$) and for each $\delta>0$ there is $C_\delta\geq 1$  such that  for all $l\geq 1$ we have
\[
\left|\frac {\bbE[T(le+y,\cdot\,;e,y)]} {l} -\bar{T}(e) \right|\leq C_\delta \, l^{-1+\beta+\delta}.
\]
Finally, $\bar C$ and $C_\delta$ can be chosen to only depend on \eqref{const} (and $C_\delta$ also on $\delta$).
\end{proposition}

Note that $\beta \in\left( \frac 12,1\right)$.  Also, see the discussion at the start of Section \ref{S4} below for the last claim.
Next we state the definition of deterministic front speeds from  \cite{LinZla}.

\begin{definition}\lb{D.5.0}
Let $f$ satisfy \textbf{(H1)} and let $e\in \bbS^{d-1}$.  If there is $c^*(e)\in\bbR$ and $\Omega_e\subseteq\Omega$ with $\bbP(\Omega_e)=1$ such that for each $\omega\in\Omega_e$ and compact $K\subseteq \{x\in \bbR^d\,|\, x\cdot e>0\}$ we have
\begin{align*}
&\lim_{t\to\infty} \,  \inf_{x\in(c^*(e)e-K)t}u(t,x,\omega;e,0)=1,\\
&\lim_{t\to\infty} \, \sup_{x\in(c^*(e)e+K)t}u(t,x,\omega;e,0)=0,
\end{align*}
then we say that $c^*(e)$ is a \textit{deterministic front speed} in direction $e$ for \eqref{1.1}.
\end{definition}

Comparison principle shows that this definition is independent of the choice of $u_{0,\calH_e^-}$ in \eqref{4.1} with $y=0$, as long as it satisfies Lemma~\ref{L.2.4} (and $c^*(e)$ is clearly unique if it exists).
It was shown in \cite[Proposition 6.2]{ZhaZla} that under the hypotheses of Proposition~\ref{P.3.1}, deterministic front speeds for $f$ exist for all $e\in\bbS^{d-1}$, and in fact they are   $c^*(e)=\bar T(e)^{-1}\in[c_0,c_1]$.  
Moreover, 
\cite[Theorems 1.3 and 1.4]{ZhaZla} 
shows that $c^*$ is Lipschitz continuous on $\bbS^{d-1}$. 

If $A$ is open convex, then we have the very useful formula \eqref{6.2}.
It will be convenient to let \eqref{6.2} be in fact the definition of $\Theta^{A,c^*}(t)$ for any continuous $c^*:\bbS^{d-1}\to(0,\infty)$ (with open convex  $A$),
and we note that then $\Theta^{A, c^*}(t)$ is also open convex for each $t\ge 0$ (openness follows from continuity of $c^*$).  
If now  $c:\bbS^{d-1}\to (0,\infty)$ is continuous and $c\leq c^*$, then clearly
$
\Theta^{A,c}(t)\subseteq \Theta^{A,c^*}(t)
$
for each $t\ge 0$.
In particular, if $c_0\leq c^*\leq c_1$ for some $c_0,c_1\in(0,\infty)$, then for all $t\geq 0$ we have
\beq\lb{3.00}
B_{c_0t}(A)\subseteq \Theta^{A,c^*}(t)\subseteq B_{c_1t}(A).
\eeq

Finally, we have the semigroup property
\beq\lb{6.4}
\Theta^{A,c^*}(t+s)=\Theta^{\Theta^{A,c^*}(t),c^*}(s)
\eeq
for all $t,s\geq 0$.
The inclusion $\supseteq$ is trivial, so
let us now consider any $x\in \Theta^{A,c^*}(t+s)$. Take any $e\in\bbS^{d-1}$, and then 
$y_e\in \partial A$ such that $y_e\cdot e=\sup_{z\in A} z\cdot e $. Then define
\[
x_e:=x-\frac{s}{t+s}(x-y_e)=y_e+\frac{t}{t+s}(x-y_e)
\]
and note that $x\in \Theta^{A,c^*}(t+s)$ implies for any $e'\in\bbS^{d-1}$ that
\[
x_e\cdot e' < y_e\cdot e' + \frac{t}{t+s} c^*(e')(t+s) = \sup_{z\in A} z\cdot e'+c^*(e')t.
\]
Hence $x_e\in {\Theta^{A,c^*}(t)}$, which together with $(x-y_e)\cdot e<c^*(e)(t+s)$ yields
\[
x\cdot e< x_e\cdot e+\frac{s}{t+s}c^*(e)(t+s)  < \sup_{z\in \Theta^{A,c^*}(t)}z\cdot e+c^*(e)s .
\]
Since this holds for all $e\in\bbS^{d-1}$, we can see that $x\in \Theta^{A,c^*}(t+s)$, and  \eqref{6.4} is proved.



\section{An Approximation Lemma}\lb{S3}


In this section we construct a perturbation ($A'$,$c'$) of ($A$,$c^*$) such that the  sets $\Theta^{A',c'}(t)$ from \eqref{6.2} satisfy an interior ball condition on a large time interval.  We will use this in the proof of Theorem \ref{T.1.1} in following section.

We say that an open set $U\subseteq \bbR^{d}$ satisfies the \textit{$r$-interior ball condition} for some $r>0$ if for any $x\in \partial U$ there is $y\in U$ such that $B_r(y)\subseteq U$ and  $x\in \partial B_r(y)$.  We also recall that $U\subseteq \bbR^d$ is \textit{strictly convex} if for all $x,y\in U$, the line segment connecting $x$ and $y$ lies in $U^0_0\cup\{x,y\}$.

\begin{lemma}\lb{L.7.3}
Let $A\subseteq\bbR^d$ be an open bounded convex set, and let $c^*:\bbS^{d-1}\to (0,\infty)$ be continuous.  If $c_0,c_1\in(0,\infty)$ are such that $c_0\le c^*\le c_1$ and $r>0$, then for any  $T\geq \frac{2r}{c_0}$ there is open convex $A'\subseteq\bbR^d$ and 
a continuous function $c':\bbS^{d-1}\to (0,\infty)$ such that 
\begin{itemize}
    \item[(i)] $c'\leq c^*$;
    \item[(ii)] $A'\subseteq B_{r}(A)$ and $\Theta^{A,c^*}(T)\subseteq B_{c_1 r/c_0}(\Theta^{A',c'}(T))$;
    \item[(iii)] $\Theta^{A',c'}(t)$ satisfies the $r$-interior ball condition for all $t\in [0,T]$.
\end{itemize}
\end{lemma}

\begin{proof}


Since $A$ is convex, by Theorem 5.4 \cite{delfour1994shape}, the signed distance function $h_A$ of $A$ (i.e., $h_A(x):=d(x,\partial A)$ if $x\in A^c$, and $h_A(x):=-d(x,\partial A)$ otherwise) is a convex function. Take any $x_0\in A$ and $\delta>0$ such that $\sup_{x\in A} \delta|x-x_0|^2< r$. Then
\[
A_1:=\big\{ x\in\bbR^d\,\big|\, h_A(x)+ \delta  |x-x_0|^2<0 \big\}
\]
 is open, convex, with $\overline{A_1}$ strictly convex and satisfying $\overline{A_1}\subseteq A\subseteq B_r(A_1)=:A'$. 
Then $A'$, which clearly satisfies the $r$-interior ball condition, also has strictly convex closure and 
\beq\lb{7.0}
A\subseteq A'\subseteq \overline{A'} \subseteq B_{r}(A).
\eeq
Since $A'$ satisfies the $r$-interior ball condition and $\overline{A'}$ is strictly convex, for each $e\in\bbS^{d-1}$ there is a unique  $x_e(0)\in \partial A'$ such that the outer  unit normal vector to $\partial A'$ at $x_e(0)$ is $e$, and $\partial A'=\bigcup_{e\in\bbS^{d-1}}\{x_e(0)\}$.
Moreover, we have
\beq\lb{3.2}
 x\cdot e < x_e(0)\cdot e\qquad \text{ for all } x\in  \overline{A'}\setminus x_e(0).
\eeq


Similarly, replacing $A$ in the above argument by $(\Theta^{A,c^*}(T))^0_{r}$, we can find an open, bounded,  convex set $A''$ satisfying the $r$-interior ball condition, having strictly convex closure, and
\beq\lb{7.1}
(\Theta^{A,c^*}(T))^0_{r}\subseteq A''\subseteq \overline{A''} \subseteq B_{r}((\Theta^{A,c^*}(T))^0_{r}) \qquad(\subseteq \Theta^{A,c^*}(T)).
\eeq
Moreover, for each $e\in\bbS^{d-1}$, there is again a unique $x_e(T)\in \partial A''$ such that the outer unit normal at $x_e(T)$ is $e$, we have $\partial A''=\bigcup_{e\in\bbS^{d-1}}\{x_e(T)\}$, as well as
\beq\lb{3.3}
 x\cdot e < x_e(T)\cdot e\qquad \text{ for all } x\in  \overline{A''} \setminus x_e(T). 
\eeq


From $T\geq \frac{2r}{c_0}$, $c^*\geq c_0$, and \eqref{3.00} we now obtain
\beq\lb{3.01}
B_{r}(A)\subseteq B_{c_0T-r}(A) = (B_{c_0T}(A))_r^0\subseteq (\Theta^{A,c^*}(T))^0_{r} \subseteq A''.
\eeq
Notice also that for any $x\in \Theta^{A,c^*}(T-\frac r{c_0})$, we have $B_{r}(x)\subseteq \Theta^{A,c^*}(T)$ due to \eqref{6.4} and \eqref{3.00}.
Therefore $\Theta^{A,c^*}(T-\frac r{c_0})\subseteq (\Theta^{A,c^*}(T))_{r}^0$, and it follows that
\beq\lb{3.02}
\Theta^{A,c^*}(T)\subseteq
B_{c_1r/c_0}(\Theta^{A,c^*}(T-c_0^{-1}r))\subseteq B_{c_1r/c_0}((\Theta^{A,c^*}(T))^0_{r})\subseteq B_{c_1r/c_0}(A'').
\eeq

\smallskip

Now define $c':\bbS^{d-1}\to \bbR$ by
\[
c'(e):=\frac{(x_e(T)-x_e(0))\cdot e} T.
\]
Then $c'>0$ because $\overline{A'}\subseteq A''$ by \eqref{7.0} and \eqref{3.01},
and it is also continuous because $\overline{A'}$ and $\overline{A''}$ are strictly convex.  Since  $A''\subseteq \Theta^{A,c^*}(T)$ and $A\subseteq A'$  by  \eqref{7.1} and \eqref{7.0}, by using \eqref{6.2}  and \eqref{3.2} we obtain
\[
x_e(0)\cdot e+c'(e)T = x_e(T)\cdot e \le \sup_{y\in  A} y\cdot e+c^*(e)T \le x_e(0)\cdot e+c^*(e)T
\]
for each $e\in\bbS^{d-1}$, so (i) holds.
 Moreover, from \eqref{6.2}, \eqref{3.2}, and \eqref{3.3} we see that
\[
\begin{aligned}
    \Theta^{A',c'}(T)
    &=\bigcap_{e\in\bbS^{d-1}}\left\{x\in\bbR^d\,\bigg|\,x\cdot e<x_e(0)\cdot e+c'(e)T\right\}\\
    &=\bigcap_{e\in\bbS^{d-1}}\left\{x\in\bbR^d\,\bigg|\,x\cdot e<x_e(T)\cdot e\right\}=A''.
\end{aligned}
\]
This, \eqref{7.0}, and \eqref{3.02} yield (ii).

It remains to show that $\Theta^{A',c'}(t)$ satisfies the $r$-interior ball condition for all $t\in [0,T]$.
For any $e\in\bbS^{d-1}$ and $t\in [0,T]$, let
\beq\lb{3.9}
x_e(t):=(1-T^{-1}t)\,x_e(0)+ T^{-1}t \,x_e(T).
\eeq
Then
\begin{align}\lb{3.4}
    x_e(t)\cdot e= x_e(0)\cdot e+c'(e)t,
\end{align}
and \eqref{3.2} and \eqref{3.3} show for all $e'\in \bbS^{d-1}\setminus\{e\}$ that
\begin{align*}
    x_e(t)\cdot e'
<(1-T^{-1}t)\, x_{e'}(0)\cdot e'+T^{-1}t \,x_{e'}(T)\cdot e'
=x_{e'}(0)\cdot e'+c'(e')t \qquad ( = x_{e'}(t) \cdot e').
\end{align*}
Therefore $x_e(t)\in \partial \Theta^{A',c'}(t)$ by \eqref{6.2}, and $x_{e}(t)\neq x_{e'}(t)$ for all $e'\in\bbS^{d-1}\setminus\{e\}$. 

Since $\Theta^{A',c'}(t)$ is bounded and convex by \eqref{6.2}, it has a supporting hyperplane for each (outer) direction $e\in\mathbb S^{d-1}$.  From $x_e(t)\in \partial \Theta^{A',c'}(t)$, \eqref{3.4}, and \eqref{6.2} we see that this hyperplane is precisely $\{x\cdot e = x_e(t)\cdot e \}$.
Since $x_{e'}(t)\cdot e < x_e(t)\cdot e$ for all $e'\in \bbS^{d-1}\setminus\{e\}$ and $x_{e'}(t)$ is continuous in $e'$ for each $t\in[0,T]$ (because it is for $t=0,T$, by strict convexity of $A',A''$), it follows that for each $\eps>0$, there is $\delta\in(0,1)$ such that if $0<|e'-e|< \delta$, then the supporting hyperplane $\{x\cdot e' = x_{e'}(t) \cdot e'\}$ contains the point $x_{e'}(t)$ satisfying $x_{e'}(t)\cdot e < x_e(t)\cdot e$ and  $|x_{e'}(t)-x_e(t)|<\eps$.  This and $x_e(t)\cdot e'<x_{e'}(t)\cdot e'$ show that the closest point to $x_e(t)$ that lies in the intersection of the two hyperplanes, which is $x_e(t)+ s_{e'} \frac{e'-(e'\cdot e)e}{|e'-(e'\cdot e)e|}$ for some $s_{e'}\in\bbR$, must have $s_{e'}\in(0,\eps)$.  But since the points  from $\{x\cdot e = x_e(t)\cdot e \}$ that satisfy $x\cdot e'\le x_{e'}(t)\cdot e'$ are precisely those with $(x-x_e(t))\cdot \frac{e'-(e'\cdot e)e}{|e'-(e'\cdot e)e|} \le s_{e'}$, and this holds for all $e'$ with $|e'-e|<\delta$, we see that $\partial \Theta^{A',c'}(t) \cap \{x\cdot e = x_e(t)\cdot e \} \subseteq B_\eps(x_e(t))$.  Taking $\eps\to 0$ shows that $\partial \Theta^{A',c'}(t) \cap \{x\cdot e = x_e(t)\cdot e \}=\{x_e(t)\}$, and so
\[
\partial \Theta^{A',c'}(t)=\bigcup_{e\in\bbS^{d-1}}\{x_e(t)\}.
\]



Now fix any $t\in [0,T]$ and $x\in \partial \Theta^{A',c'}(t)$, and let $e\in\bbS^{d-1}$ be such that $x_e(t)=x$.
Since $A'$ and $A''$ satisfy the $r$-interior ball condition, there are $y_0,y_T$ such that $B^0:=B_r(y_0)$ and $B^T:=B_r(y_T)$ satisfy $B^0\subseteq A'$, $B^T\subseteq A''$, $x_e(0)\in \partial B^0$, and $x_e(T)\in \partial B^T$.
If now
\[
B^{t}:=B_r((1-T^{-1}t)\,  y_0+ T^{-1}t \, y_T),
\]
then  \eqref{3.9} shows that $x=x_e(t)\in\partial B^{t}$. 
It therefore remains to show  that $B^{t}\subseteq \Theta^{A',c'}(t)$.
For any $z\in B^{t}$, there are $z_0\in B^0$ and $z_T\in B^T$ such that $z=(1-T^{-1}t)\,z_0+T^{-1}t\,z_T$.  It follows from \eqref{3.2} and \eqref{3.3} that for any $e'\in\bbS^{d-1}$ we have
\begin{align*}
    z\cdot e'
    <(1-T^{-1}t)\, x_{e'}(0)\cdot e'+ T^{-1}t\, x_{e'}(T)\cdot e'
    = x_{e'}(0)\cdot e' +c'(e')t 
    =\sup_{y\in  A'}y\cdot e'+c'(e')t,
\end{align*}
and hence $z\in \Theta^{A',c'}(t)$ by \eqref{6.2}. Thus $B^{t}\subseteq \Theta^{A',c'}(t)$, finishing the proof.
\end{proof}

\section{Proof of Theorem \ref{T.1.1}}\lb{S4}

We will do the proof simultaneously for $f$ satisfying \textbf{(H2)} and having finite range of dependence (then we assume this range to be at most $\rho\in[1,\infty)$), and for $f$ satisfying \textbf{(H2)--(H4)}.  This is because Proposition \ref{P.3.1} applies in both these cases, with  the definitions \eqref{b.1'} and \eqref{b.1''}, respectively (we will use these below).
We also let $c^*$ be the deterministic front speed for \eqref{1.1}.

Before we start, for any solution $u:[0,\infty)\times \bbR^d\to [0,1]$ to \eqref{1.1} and any $0<\eta<\theta<1$, we let the {\it width of the transition zone} of $u$ from $\eta$ to $\theta$ (at any time $t\geq 0$) be (see  \cite{ZlaInhomog})
\beq\lb{d.2.1}
L_{u,\eta,\theta}(t,\omega):=\inf \left\{L>0\, \big| \, \Gamma_{u,\eta}(t,\omega) \subseteq B_L\left(\Gamma_{u,\theta}(t,\omega) \right) \right\}.
\eeq
It follows from  Remark 2 after \cite[Definition 2.3]{ZhaZla} and \cite[Lemma 2.4]{ZhaZla} that if $f$ satisfies \textbf{(H2)}, then there are $\mu_*,\kappa_*>0$ 
such that if  $u$ solves \eqref{1.1} with some $\omega\in\Omega$ and initial data satisfying Lemma~\ref{L.2.4} for some $S\subseteq\bbR^d$, then
\begin{align}
\sup_{t\geq0 \,\&\, \eta\in (0,1-\theta^*)} \frac{ L_{u,\eta,1-\theta^*}(t) } { 1+ |\ln \eta| } & \le \mu_*^{-1}, \lb{2.5}
\\    \inf_{\substack{(t,x)\in[\kappa_*,\infty)\times \bbR^d\\  u(t,x)\in [\theta^*,1-\theta^*]}}u_t(t,x)\geq \mu_*. \notag
\end{align}
We will in fact only need this in the first part of this proof, for $S$ being half-spaces (so for the solutions from \eqref{4.1}).

Moreover, it follows from the above results in \cite{ZhaZla} that there is $\eta_*=\eta_*(M, \theta_1,m_1,\alpha_1)>0$ such that $\mu_*,\kappa_*$ can be chosen to depend only on $M, \theta_1,m_1,\alpha_1$ from {\bf (H1)} and $\xi>0$ from \eqref{1.111}.
Similarly, $\bar C$ and $C_\delta$ in Proposition \ref{P.3.1} can be chosen to depend only on \eqref{const} 
(and $C_\delta$ also depends on $\delta$)
 because they depend on the constants from (2.10) in \cite{ZhaZla} (plus $\rho$ in the finite range of dependence setting), which is \eqref{const} without $\xi,\rho$ and also with $\mu_*,\kappa_*,m_2,\alpha_2,m_4'$.  But when we assume \textbf{(H2)}, we can simply let $m_2:=1$ and $\alpha_2:=0$ in \cite{ZhaZla}  because $1+|\ln\eta|\le \eta^{-1}$ for $\eta\in(0,1)$; and when we assume also \textbf{(H4)}, in which case $m_4'$ also plays a role in \cite{ZhaZla} , we can let $m_4':=\infty$.

In the rest of this section, constants that include $C$ {\it will again depend on \eqref{const}},
while any other dependence will be explicitly declared in the notation (e.g., $C_{\eps,T}'$ also depends on $\eps,T$).
These constants may also {\it vary from one expression to the next}. 

We are now ready for the proof of Theorem \ref{T.1.1}, which we split into two main parts.  Without loss, we will assume that $T_0\ge 1$.


\subsection{Proof of the ``Upper Bound''}\lb{ss4.1}
In this part we will prove \eqref{6.1} below for all small $\eps>0$.  
Let us pick 
\beq\lb{6.11}
{\sigma'}:=\min\left\{\frac{1-\be}4,{\nu}\right\}=2\sigma, 
\eeq
and some $\eps_0\in(0,\frac 12)$ such that
\beq\lb{eps}
\max \left\{ \left(1+|\ln(\theta-\eps_0^{1/m_1})|\right)\mu_*^{-1}\eps_0^{1-2{\sigma'}},\,  \left((\theta^*)^{-1} +4C' \right)\eps_0^{{\sigma'}}\right\} \leq 1,
\eeq 
with $C'\ge 1$ to be determined later.  Note that this $\eps_0$ depends only on \eqref{const} and $\nu,\theta$. 



Fix any $y\in \partial A$ and ${e_y}\in\bbS^{d-1}$ such that $A\subseteq \calH_{{e_y}}^-+y$ (such $e_y$ always exists because $A$ is convex, and we call it an outer normal to $\partial A$ at $y$). Then let 
\[
{v}^\eps_{y}(t,x,\omega):=u(t,x,\omega;e_y,\eps^{-1}y),
\]
with the right-hand side function defined in \eqref{4.1}.
If we now let $u^\eps(t,x,\omega):=u_\eps(\eps t,\eps x,\omega)$, then Lemma \ref{L.2.2} and Lemma \ref{L.2.8} yield 
\beq\lb{661'}
{v}^\eps_{y}(\tau_\eps,\cdot,\omega)\geq (1-\eps^{1/m_1})\chi_{\calH_{e_y}^-+\eps^{-1}(y+\eps^{\nu}e_y)}\geq u^\eps(0,\cdot,\omega)-\eps^{1/m_1}
\eeq
with
\[
\tau_\eps:=\kappa_0+2c_0^{-1} \eps^{\nu-1}+ D_1 \eps^{(1-m_1)/m_1}
\]
(then also $\tau_\eps\leq C\eps^{{\sigma'}-1}$ for some $C>0$ due to \eqref{6.11}
and $\nu\geq \sigma'$). 

It follows from \eqref{661'} and the last claim in Lemma 3.7 in \cite{ZhaZla} with $f_2=f_1=f$ (this extends Lemma 2.9 in \cite{ZhaZla} from initial data approximating characteristic functions of balls to those in \eqref{4.1}, which instead approximate characteristic functions of half-spaces) that if we extend $f$ to $\bbR^d\times (1,\infty)\times\Omega$ by 0, then with $M_*:=\frac{1+M}{\mu_*}$ we have that 
\[
v_y^\eps((1+M_*\eps^{1/m_1})t+\tau_\eps,x,\omega)+\eps^{1/m_1}
\]
is a supersolution to \eqref{1.1} for $(t,x)\in(\kappa_*,\infty)\times\bbR^d$.
Hence if we let $\tau_\eps':=\tau_\eps+(1+M_*\eps^{1/m_1})\kappa_*$ and use $(v_y^\eps)_t\geq 0$ (by Lemma \ref{L.2.4}) and \eqref{661'}, we obtain from the comparison principle that
\[
w^\eps_{y}(t,x,\omega):=v_y^\eps((1+M_*\eps^{1/m_1})t+\tau_\eps',x,\omega)+\eps^{1/m_1}\geq u^\eps(t,x,\omega),
\]
 for all $(t,x)\in(0,\infty)\times\bbR^d$.  Therefore, 
\beq\lb{661}
w_{\eps,y}(\cdot,\cdot,\omega):=w^\eps_{y}(\eps^{-1}\cdot,\eps^{-1}\cdot,\omega)\geq u_\eps(\cdot,\cdot,\omega)
\eeq
on $(0,\infty)\times\bbR^d$.
We can now use this estimate to prove \eqref{6.1}.

Let us first obtain a crude $\omega$-uniform bound.  Corollary~\ref{C.2.1} yields
\[
\Gamma_{v^\eps_{y},1-\theta^*}(t,\omega)\subseteq \Gamma_{v^\eps_{y},1-\theta_1}(t,\omega)\subseteq \calH_{{e}}^- +\eps^{-1}y+(R_0+\kappa_1+c_1t){e_y},
\]
and so from $R_0+\kappa_1+c_1[(1+M_*\eps^{1/m_1})t+\tau_\eps']\leq C(t+\eps^{{\sigma'}-1})$ for some $C> 0$, 
we obtain
\[
\Gamma_{w_{\eps,y},1-\theta^*}(t,\omega)\subseteq \calH_{{e}}^- +y +C(t+\eps^{{\sigma'}}){e_y}
\]
for all $t\ge 0$.
From \eqref{2.5} and \eqref{eps} we see that
$\sup_{t\geq 0}L_{v^\eps_{y},\theta-\eps^{1/m_1},1-\theta^*}(t)\leq \eps^{2{\sigma'}-1}$, hence
\beq\lb{671}
\sup_{t\geq 0}L_{w_{\eps,{y}},\theta,1-\theta^*}(t)\leq \eps \sup_{t\geq 0}L_{v^\eps_{y},\theta-\eps^{1/m_1},1-\theta^*}(t)\leq \eps^{2{\sigma'}}.
\eeq
So in view of  \eqref{661}, for all $t\geq 0$ we get
\beq\lb{6.6}
\Gamma_{u_\eps,\theta}(t,\omega)\subseteq \Gamma_{w_{\eps,y},\theta}(t,\omega)\subseteq \calH_{{e}}^- +y+C(t+\eps^{{\sigma'}}){e_y}.
\eeq 
Since this holds for all $y\in \partial A$ and all normal directions ${e_y}$ at $y$, there exists $C>0$ such that for all $t\geq 0$ and $\omega\in\Omega$ we have
\beq\lb{6.6'}
\Gamma_{u_\eps,\theta}({t},\omega)\subseteq B_{C(t+\eps^{{\sigma'}})}(A).
\eeq 


Now fix any ${T}\geq \eps^{{\sigma'}}\ (>\eps)$, so we have
\beq\lb{4.3}
\Gamma_{u_\eps,\theta}({T},\omega)\subseteq B_{\tilde C {T}}(A),
\eeq 
with $\tilde C:=2\max\{C,c_1\}$.
Next, take any $\bar{y}\in B_{\tilde C{{T}}}(A)\setminus \overline{\Theta^{A,c^*}({T})}$, let $y$ be the unique projection of $\bar{y}$ onto $\partial A$, let $e_y:=\frac{\bar{y}-y}{|\bar{y}-y|}$ (which is then an outer normal to $\partial A$ at $y$), and define ${v}^\eps_{y}$, $w^\eps_{y}$, and $w_{\eps,y}$ as above.
Then the definition of $\Theta^{A,c^*}(\cdot)$ yields
\beq\lb{6.7}
c^*(e_y){T} \le |\bar{y}-y|\le  \tilde C {T}.
\eeq
Consider the arrival times
\begin{align*}
T^w_{\eps}(\bar{y},\omega)&:=\inf\{t\geq 0\,|\, w_{\eps,y}(t,\bar{y},\omega)\geq 1-\theta^*\},\\
T^v_{\eps}(\bar{y},\omega)&:=\eps^{-1}\inf\{t\geq 0\,|\, v^\eps_{y}( t,\eps^{-1}\bar{y},\omega)\geq 1-\theta^*\},
\end{align*}
both of which are $\leq CT$ for some $C>0$ by \eqref{6.7} and Lemma \ref{L.2.2}.
Since \eqref{6.11} and \eqref{eps} yield $\eps^{1/m_1}\leq {\theta^*}$ (and so $1-\theta^*-\eps^{1/m_1}\ge\theta_2$), Lemma \ref{L.2.2}, $\sigma'< \frac1{m_1}$, and  $\eps\tau_\eps'\leq C\eps^{{\sigma'}}$ imply
\beq\lb{669}
T^v_{\eps}(\bar{y},\omega)\le T^w_{\eps}(\bar{y},\omega) +C\eps^{1/m_1}T + \eps\tau_\eps' + \eps \kappa_0 \leq T^w_{\eps}(\bar{y},\omega)+ C(1+ T) \eps^{{\sigma'}}.
\eeq

Next, after applying Proposition {\ref{P.3.1}} to $v^\eps_y$ with $\delta:={\sigma'}$ and $l:=|\bar{y}-y|=(\bar{y}-y)\cdot e_y$, and using $c^*(e_y)=\bar T(e_y)^{-1}$ and $\beta\leq 1-4{\sigma'}$, we get
\[
\left|\frac{\bbE[T^v_{\eps}(\bar{y},\cdot)]}{|\bar{y}-y|}-\frac{1}{c^*(e_y)}\right|
\leq C \left(\eps^{-1}{|\bar{y}-y|}\right)^{-3{\sigma'}}
\]
(here we call the constant $C_\delta=C_{{\sigma'}}$ just $C$).
This and \eqref{6.7} yield
\beq\lb{6613}
{T}-\bbE[T^v_{\eps}(\bar{y},\cdot)]\leq   C |\bar{y}-y|^{1-3{\sigma'}} \eps^{3{\sigma'}}\leq  
C(1+T)\eps^{3{\sigma'}}.
\eeq
Using \eqref{6.7} again, it follows from Proposition \ref{P.3.1} that for all $\lambda\geq 0$,
\beq\lb{6610}
\begin{aligned}
\bbP[|T^v_{\eps}(\bar{y},\cdot)-\bbE[T^v_{\eps}(\bar{y},\cdot)]|>\eps\lambda]&\leq 2\exp\left(- C^{-2}\lambda^2( \eps^{-1}|\bar{y}-y|)^{-2\beta}\right)\\
&\leq 2\exp\left(- C^{-2}\lambda^2{T}^{-2\beta}\eps^{2\beta}\right) .
\end{aligned}
\eeq
Now take $\lambda:=C {T}^\beta \eps^{-\beta-{\sigma'}}$ with $C$ from the last expression, and then $\lambda\leq  CT\eps^{2{\sigma'}-1}$ by $\beta\leq 1-4{\sigma'}$ and $T\ge\eps^{{\sigma'}}$. Hence
 \eqref{6613} and \eqref{6610} show that there is $C>0$ such that 
\[
\begin{aligned}
\bbP[T^v_{\eps}(\bar{y},\cdot)\leq {T}- C(1+T)\eps^{2{\sigma'}}]
\leq 2\exp\left(-\eps^{-2{\sigma'}}\right).
\end{aligned}
\]
Using \eqref{669} yields, with some $C>0$ and  $C_{\eps,T}:=C(1+T)\eps^{{\sigma'}}$,
\beq\lb{4.7}
\begin{aligned}
\bbP[w_{\eps,y}({T}- C_{\eps,T},\bar{y},\cdot)\geq 1-\theta^*]=\bbP[T^w_{\eps}(\bar{y},\cdot)\leq {T}- C_{\eps,T}]\leq 2\exp\left(-\eps^{-2{\sigma'}}\right).
\end{aligned}
\eeq

Let now $ r:=\eps^{2{\sigma'}}$ ($\in(\eps,T)$ because ${\sigma'}\in (0,\frac{1}{8})$), and note that \eqref{671} implies
\beq\lb{6614}
L_{\eps}:=r+\sup_{z\in\partial A}\sup_{t\ge 0}L_{w_{\eps,z},\theta,1-\theta^*}(t,\omega)\leq  2\eps^{2{\sigma'}}.
\eeq
Next let $G_{\eps,T}\subseteq B_{\tilde C{{T}}}(A)\backslash \Theta^{A,c^*}({T})$ be some set containing one point from each cube in $\bbR^d$ with side length $rd^{-1/2}$ and all vertices in $rd^{-1/2}\bbZ^d$ that has a non-empty intersection with $B_{\tilde C{{T}}}(A)\backslash \Theta^{A,c^*}({T})$.  Note that then  $B_{\tilde C{{T}}}(A)\backslash \Theta^{A,c^*}({T}) \subseteq B_{r}(G_{\eps,T})$. 

Let us now consider any $\bar{y}\in G_{\eps,T}$.
If we have
$w_{\eps,y}(t,x,\omega)\geq\theta$ for some $x\in B_{r}(\bar{y})$ and $t\geq 0$, then \eqref{d.2.1} shows that there is $x'\in B_{L_\eps}(\bar{y})$ such that $w_{\eps,y}(t,x',\omega)\geq 1-\theta^* $. Applying Lemma \ref{L.2.2} to $v_y^\eps$ then implies $w_{\eps,y}(t+ 2c_0^{-1}L_\eps+\eps\kappa_0 ,\bar{y},\omega)\geq 1-\theta^* $. Since 
\[
C_{\eps,T}+2c_0^{-1}L_\eps+\eps\kappa_0\le  C'(1+T)\eps^{{\sigma'}}=:C_{\eps,T}'
\]
 by \eqref{6614} (with some $C'\ge 1$, that will then also be the number in \eqref{eps}), from \eqref{4.7} we get  
\beq\lb{6.8}
\bbP\left[  w_\eps({T}- C_{\eps,T}',x,\cdot)\geq \theta \text{ for some } x\in B_{r}(\bar{y}) \right]\leq 2\exp\left(-\eps^{-2{\sigma'}}\right)
\eeq
(with the understanding that this probability is $0$ when ${T}- C_{\eps,T}'<0$).
Then \eqref{4.3}, \eqref{6.8},  $ w_{\eps,y}\geq u_\eps$, $(w_{\eps,y})_t\geq 0$, and the fact that $|G_{\eps,T} |\leq C_A{T}^dr^{-d}$ for some $C_A>0$ (depending only on the diameter of $A$ and \eqref{const}) yield 
\beq\lb{6615}
\begin{aligned}   
\bbP\left[\bigcup_{t\in [0,{T}- C_{\eps,T}']}\Gamma_{u_\eps,\theta}(t,\cdot)\not\subseteq \Theta^{A,c^*}({T}) \right]
&\leq \sum_{\bar{y}\in G_{\eps,T}}\bbP\left[ w_\eps({T}- C_{\eps,T}',x,\cdot)\geq \theta\text{ for some } x\in B_{r}(\bar{y}) \right]\\
&\leq 2C_A{T}^d r^{-d}\exp\left(-\eps^{-2{\sigma'}}\right).
\end{aligned}
\eeq

From \eqref{eps} and $T_0\ge 1$ we now have $C_{\eps,T_0}'<  \frac{T_0}2$.  So for any $t\in [C_{\eps,T_0}',T_0]$, there is a unique  $T\in(t,2t)$ such that $t={T}-C_{\eps,T}'$.  Then $C_{\eps,T}'\leq 2 C_{\eps,T_0}'$ and so 
\[
\Theta^{A,c^*}({T})\subseteq B_{3c_1 C_{\eps,T_0}'}(\Theta^{A,c^*}(t-C_{\eps,T_0}'))
\]
 by $c^*\leq c_1$. 
Then \eqref{6615} yields
\[
    \bbP\left[\Gamma_{u_\eps,\theta}(s,\cdot)\not\subseteq B_{3 c_1 C_{\eps,T_0}'}  \left(\Theta^{A,c^*}(s )\right) \text{ for some } s\in[t-C_{\eps,T_0}',t] \right]    
    \leq 2^{d+1}C_A T_0^d \eps^{-2d{\sigma'}}\exp\left(-\eps^{-2{\sigma'}}\right),
\]
and so from $\lceil T_0(C_{\eps,T_0}')^{-1} \rceil \le 2\eps^{-{\sigma'}}$ we obtain
\[
    \bbP\left[ \Gamma_{u_\eps,\theta}(s,\cdot)\not\subseteq B_{3 c_1 C_{\eps,T_0}' }\left(\Theta^{A,c^*}(s )\right) \text{ for some } s\in[0,T_0] \right]    
    \leq 2^{d+2}C_A T_0^{d}  \eps^{-(2d+1){\sigma'}}\exp\left(-\eps^{-2{\sigma'}}\right).
\]    
If we make $\eps_0>0$ smaller yet, depending on the constants mentioned after \eqref{eps} as well as $A$ and $T_0$, then for all $\eps\in(0,\eps_0)$ this shows 
\beq\lb{6.1}
\bbP\left[\Gamma_{u_\eps,\theta}(t,\cdot)\subseteq B_{\eps^{\sigma}}\left(\Theta^{A,c^*}(t)\right)\text{ for all }t\in[0,T_0]\right]\geq 1- \exp\left(-\eps^{-2\sigma}\right).
\eeq

\subsection{Proof of the ``Lower Bound''}
The second part of this proof is considerably more involved than the first. This is because a lower bound for the solution $u_\eps$ is needed here, but the solutions $u(\cdot,\cdot,\cdot;e,y)$ with front like initial data cannot serve as global barriers from below. We overcome this problem by using them as approximate local barriers on short time intervals, making use of Lemma \ref{L.7.3} in the process.

 Our goal is now to prove a counterpart to \eqref{6.1}, namely
\beq\lb{6.1''}
\bbP\left[  \left(\Theta^{A,c^*}(t)\right)^0_{\eps^{\sigma}} \subseteq \Gamma_{u_\eps,\theta}(t,\cdot) \text{ for all }t\in[C_{\theta,A}\eps,T_0]\right]\geq 1- \exp\left(-\eps^{-2\sigma}\right)
\eeq
for all $\eps\in(0,\eps_0)$, with some $C_{\theta,A}$ and with $\eps_0>0$ depending on the constants mentioned after \eqref{eps} as well as $A$ and $T_0$.  Of course, this will then finish the proof.

We will simplify our task a little, so we only have to study $(1-\theta^*)$-level sets of a special solution $\tilde{u}^\eps$  to \eqref{1.1} with initial data $\tilde{u}^\eps_0$ satisfying Lemma \ref{L.2.4} with $S=\eps^{-1}(A_{ \eps^{\nu}}^0)=:A^\eps$.  
We again let $\tilde{u}_\eps(t,x,\omega):=\tilde{u}^\eps(\frac t\eps,\frac x\eps,\omega)$, and claim that for some $\tau_0=\tau_0(M, \theta_1,m_1,\alpha_1,\theta,A)>0$ we have
\beq\lb{4.4}
\Gamma_{\tilde{u}_\eps,1-\theta^*}(t-\tau_0\eps,\omega)\subseteq \Gamma_{u_\eps,\theta}\left(t,\omega\right)
\eeq
 for all $t\geq \tau_0\eps$.
Indeed,
let $U:[0,\infty)\to[0,1]$ be a solution to $U'=F_0(U)$ with initial data $U(0)=1-\theta_1$. Since $F_0(u)>0$ for all $u\in [1-\theta_1,1)$, there is $\tau_1=\tau_1(m_1,\alpha_1)>0$ such that $U(\tau_1)\geq 1-\frac{1}{2}\theta^*$.  
It follows from Lemma \ref{L.2.1} with $u_1(t,x):=U(t)$, $u_2:=u^\eps$, and $r:=2\sqrt{Md\,}\tau_1+ \sqrt{d/M\,}\ln\frac{4d}{\theta^*}$
that
\[
u^\eps(\tau_1,\cdot,\omega)\ge U(\tau_1) - 2de^{\sqrt{M/d\,} \left( -r + 2\sqrt{Md\,}\tau_1 \right)}\geq 1-\theta^*
\]
on $ (A^\eps)_{r}^0$ (which is non-empty if $\eps_0>0$ is small enough, depending on $A,\nu$).
Next let
\[
\tau_2:= \tau_1+2c_0^{-1} r' +2c_0^{-1}R_0+\kappa_0,
\]
where $B_{r'}\left((A^\eps)_{r}^0 \right)\supseteq A^\eps$ for all small enough $\eps>0$ (such $r'=r'(A,r)$ exists because  $A$ is convex and hence $\partial A$ is Lipschitz). 
Then $u^\eps(\tau_2,\cdot,\omega)\geq 1-\theta^*$ on $B_{R_0}(A^\eps)$ by Lemma \ref{L.2.2}, so $u^\eps(\tau_2,\cdot,\omega)\geq \tilde{u}^\eps(0,\cdot,\omega)$. Thus for all $(t,\omega)\in [0,\infty)\times\Omega$ we obtain
\[
\Gamma_{\tilde{u}_\eps,1-\theta^*}(t,\omega)\subseteq \Gamma_{u_\eps,1-\theta^*}\left(t+\tau_2\eps,\omega\right).
\]
When $\theta\leq 1-\theta^*$, this immediately yields \eqref{4.4} with $\tau_0:=\tau_2$ .
When $\theta\in (1-\theta^*,1)$, this and Lemma \ref{L.2.8} yield \eqref{4.4} with $\tau_0:=\tau_2+1+D_1(1-\theta)^{1-m_1}$.



Let now $\sigma'$ be from \eqref{6.11}. We  claim that \eqref{6.1''} will follow once we show that there is $\tilde C>0$ such that for all $T_0\geq 1$ and $\eps>0$ small enough (depending on the constants after \eqref{eps} and $A,T_0$) we have
\beq\lb{6.1'}
\bbP\left[(\Theta^{\eps A^\eps,c^*}(t))^0_{ \tilde C T_0 \eps^{\sigma'}}\subseteq \Gamma_{\tilde{u}_\eps,1-\theta^*}(t,\cdot)\text{ for all }t\in [0,T_0]\right]\geq 1- \exp\left(-\eps^{-{\sigma'}}\right).
\eeq
Indeed,  for all small $\eps>0$ we have
\[
\left( \Theta^{A,c^*}(t) \right)^0_{2 \eps^{{\sigma'}}} \subseteq \Theta^{\eps A^\eps,c^*}(t-\tau_0\eps)
\]
for all $t\ge \tau_0\eps$ due to convexity of $A$, \eqref{6.2},  and \eqref{6.11}.  This and \eqref{4.4} now show that if  $(\Theta^{\eps A^\eps,c^*}(t))^0_{ \tilde C T_0 \eps^{\sigma'}}\subseteq \Gamma_{\tilde{u}_\eps,1-\theta^*}(t,\omega)$ for all $t\in[0,T_0]$, then
\[
\left( \Theta^{A,c^*}(t) \right)^0_{ (2+\tilde CT_0)\eps^{{\sigma'}}}\subseteq \Gamma_{{u}_\eps,\theta}(t,\omega)
\]
for all $t\in[\tau_0\eps,T_0]$.
So again, if we make  $\eps_0>0$ smaller yet, depending on the constants mentioned after \eqref{eps} as well as $A$ and $T_0$, then \eqref{6.1'} will indeed imply \eqref{6.1''} with $C_{\theta,A}:=\tau_0$.


So let us now prove \eqref{6.1'}.  In the proof, we will write $u_\eps$ and $A$ in place of $\tilde{u}_\eps$ and $\eps A^\eps$ (so $(u_\eps)_t\geq 0$),  
and  denote
\beq\lb{4.9}
{\sigma''}:=\frac{1-\beta}{2(2-\beta)}\in \left({\sigma'},\frac{1}{6} \right) \qquad\text{and}\qquad r_\eps:=\eps^{{\sigma''}}
\eeq
(recall that $\beta\in (\frac12,1)$).
Let us also pick $\eps_0\in(0,\frac 12)$ such that
\beq\lb{eps'}
\max\left\{2c_0^{-1}\eps_0^{\sigma''-\sigma'},\,D_2(1+\kappa_0+c_1(1+4c_0^{-1}))\eps_0^{\sigma''}\right\}\leq 1
\eeq
(where $\kappa_0$ is from Lemma \ref{L.2.2} and $D_2$  from Corollary \ref{L.2.9}); we will need to further decrease $\eps_0$ later.
For any $u:[0,\infty)\times\bbR^d\times\Omega\to [0,1]$, let us denote  by
\[
T_{u}(x,\omega):= \inf\{t\geq 0\,|\, u(t,x,\omega)\geq 1-\theta^*\}
\]
the arrival time at $x\in\bbR^d$.

Now we fix any $\eps\in(0,\eps_0]$ and ${T}\in [\eps^{{\sigma'}},T_0]$,
and pick $A',c'$ as in Lemma \ref{L.7.3} with $r=r_\eps$   (then ${T}\geq \frac{2r}{c_0}$ by \eqref{eps'}). 
Then let $\Theta^k_{\eps,T}:=\Theta^{A',c'}(kr_\eps^2)$ for each $k\in \bbN$, and  
\[
t_k(\omega):=\inf\left\{t\geq 0\,|\,u_\eps(t,\cdot,\omega)\geq (1-\theta^*)\chi_{\Theta^k_{\eps,T}}\right\}
\]
for each $\omega\in\Omega$.
Note that from Lemma \ref{L.2.2} and Lemma \ref{L.7.3}(ii) we obtain
\beq\lb{6.15}
t_0(\cdot)\leq 2c_0^{-1}r_\eps+\kappa_0\eps.
\eeq
Let $K:= \lceil {T}r_\eps^{-2} \rceil$, so that clearly
$\Theta^{A',c'}({T})\subseteq \Gamma_{u_\eps,1-\theta^*}(t_K(\omega),\omega)$ for all $\omega\in \Omega$.
Our goal is now to prove \eqref{6612} below, which is a high-probability upper bound  on $t_{k+1}(\cdot)-t_k(\cdot)$ for each $k=0,1,\dots,K-1$. 
Adding these will then yield a high-probability upper bound on $t_K(\cdot)$, and therefore also the estimate \eqref{6.87} below, which is very close to \eqref{6.1'} for the single time $T$ instead of all $t\in[0,T_0]$.
We will then upgrade this to \eqref{6.1'}.


Fix any $x_0\in {\Theta^{k+1}_{\eps,T}}\setminus {\Theta^k_{\eps,T}}$ and $\omega\in\Omega$.  Since $\Theta^k_{\eps,T}$ is convex,  there is $x_1\in \partial{\Theta^k_{\eps,T}}$ such that 
$
d(x_0,{\Theta^k_{\eps,T}})=|x_0-x_1|,
$
and $e:=\frac{x_0-x_1}{|x_0-x_1|}$ is an outer normal to $\partial \Theta^k_{\eps,T}$ at $x_1$. Then
\beq\lb{4.8}
d_0:={|x_0-x_1|}\leq c'(e){r_\eps^2}\leq c^*(e){r_\eps^2}
\eeq
by \eqref{6.4} and Lemma \ref{L.7.3}(i).
Since $\Theta^{k}_{\eps,T}$ satisfies the $r_\eps$-interior ball condition by Lemma \ref{L.7.3}(iii), $e$ is the unique outer normal to $\partial \Theta^k_{\eps,T}$ at $x_1$ and 
\[
B_{r_\eps}(x_1-r_\eps e)\subseteq \Theta^k_{\eps,T}
\subseteq\Gamma_{u_\eps,1-\theta^*}(t_k(\omega),\omega).
\] 
So if we let
$
w_{k}^\eps(t,x,\omega):=u_\eps(t_k+\eps t,x_0+\eps x,\omega)
$,
then clearly
\beq\lb{6.9}
B_{\eps^{-1}r_\eps}\left(-\eps^{-1}(d_0+r_\eps)e\right) \subseteq\Gamma_{w_{k}^\eps,1-\theta^*}(0,\omega).
\eeq

Let us now define $d_1:=c_1{{r_\eps^2}}+D_2((1+\kappa_0)\eps+4c_0^{-1}c_1r_\eps^2)$, with $D_2\geq 1$ from Corollary \ref{L.2.9}.  Then $d_1> \max\{r_\eps^2,d_0\}$ by \eqref{4.8}, and 
$d_1< \min\{ r_\eps, C r_\eps^2\}$ for some $C>0$  by $2\sigma''< 1$ and  \eqref{eps'}.  We also let
\[
d_2:=\frac{d_1^2+d_0^2+2d_0r_\eps}{2(d_0+r_\eps)},
\]
so then $d_1-d_2=\frac{(d_1-d_0)(2r_\eps-d_1+d_0)}{2(d_0+r_\eps)}>0$ and $d_2-d_0=\frac{d_1^2-d^2_0}{2(d_0+r_\eps)}>0$. Hence
\beq\lb{6.13}
0\le d_0< d_2< d_1\leq \min\{r_\eps,  Cr_\eps^2\} \qquad\text{ and }\qquad d_2-d_0
 \leq Cr_\eps^{3},
\eeq
with some $C>0$.
We then have
\beq\lb{6.9'}
\left\{x\in\bbR^d \,\Big|\, x\cdot e< -{\eps}^{-1}{d_2} \right\}\cap B_{{\eps}^{-1}d_1}(0) \subseteq B_{{\eps}^{-1}r_\eps}\left(-{\eps}^{-1}(d_0+r_\eps)e\right),
\eeq
which follows from \eqref{6.13} and the fact that the spherical cap on the left has axis $e$ and the radius of its base is $\sqrt{d_1^2-d_2^2}$, which equals $\sqrt{r_\eps^2-(r_\eps+d_0-d_2)^2}$ due to the definition of $d_2$.


Now let
\[
v(\cdot,\cdot,\omega):=u(\cdot,\eps^{-1} x_0+\cdot,\omega;e,\eps^{-1}(x_0-d_2e))
\]
where $u$ is from \eqref{4.1}.
Then $v$ and $w_{k}^\eps$ both satisfy \eqref{1.1} with $f$ shifted in space by $\frac{x_0}\eps$,
and  $\supp \,v(0,\cdot,\omega)\subseteq \calH_{e}^{-}+(R_0- \frac{d_2}\eps) e$.  This, \eqref{6.9}, \eqref{6.9'}, and Lemma \ref{L.2.2} yield
\[
 v(0,\cdot,\omega)\leq {w}_k^\eps (\tau_3,\cdot,\omega)
 \]
 on $B_{\eps^{-1}d_1}(0)$, where  $\tau_3:=\frac{2R_0}{c_0}+ \kappa_0$.
Since $v(0,\cdot,\omega)\geq (1-\theta^*)\chi_{\calH_e^--{\eps^{-1}d_2}e} $, from Lemma \ref{L.2.2}  we also obtain
$
T_{v}(0,\omega)\leq 2(\eps c_0)^{-1}d_2+\kappa_0,
$ so the definition of $d_1$, \eqref{4.8}, and  \eqref{6.13} yield 
\[
\eps^{-1 }d_1\geq D_2\left(1+\kappa_0+{4(\eps c_0)^{-1}c_1{r_\eps^2}}\right)\geq D_2(1+ T_v(0,\omega)),
\]
provided $\eps_0>0$ is small enough (depending on \eqref{const}) so that $d_2\le c^*(e)r_\eps^2+C r_\eps^3 \le 2c_1r_\eps^2$.
So
Corollary \ref{L.2.9} with 
\[
u_1:=v(\cdot,\cdot,\omega),\quad u_2:=w_{k}^\eps(\cdot,\cdot,\omega),\quad t_0:=\tau_3,\quad\text{ and }\quad R:=\eps^{-1}d_1,
\]
yields 
\beq\lb{4.11}
T_{v}(0,\omega)\geq T_{w_{k}^\eps}(0,\omega)-\tau_3-\kappa_0.
\eeq

We next apply both claims in  Proposition {\ref{P.3.1}}, with $\delta:={\sigma'}$ and $l:=\eps^{-1}d_2$ (also recall that $\bar{T}(e)=c^*(e)^{-1}$), to obtain
\beq\lb{4.13}
\bbP\left[\left|T_{v}(0,\cdot)-(\eps c^*(e))^{-1} {d_2} \right|\geq {C}\left(\eps^{-1}d_2 \right)^{\beta+{\sigma'}}+\lambda\right]\leq 2\exp\left(-\bar{C}^{-2}\lambda^2\left(\eps^{-1}d_2\right)^{-2\beta}\right)
\eeq
for some $C>0$ and all $\lambda\geq 0$.
Let us then take $\lambda:=\bar{C} (\eps^{-1}{d_2})^{\beta}\eps^{-{\sigma'}}$.
We get from \eqref{4.9} and \eqref{6.13} that
\beq\lb{4.13'}
(\eps^{-1}d_2)^{\beta+{\sigma'}} \le (\eps^{-1}d_2)^{\beta}\eps^{-{\sigma'}} \leq C\eps^{3{\sigma''}-1}
\eeq
because \eqref{4.9}  yields
\[
3\sigma''+\sigma'-2\sigma''\beta\leq \sigma'' (4-2\beta)\le 1-\beta.
\]
Then \eqref{4.13} and 
$d_2\leq c^*(e) r^2_\eps+Cr_\eps^{3} $
show that with some $C>0$ we have
\beq\lb{6.5}
\bbP\left[T_{v}(0,\cdot)\geq {\eps}^{2{\sigma''}-1}+ {C} {\eps}^{3{\sigma''}-1}\right]\leq 2\exp\left(-{\eps^{-2{\sigma'}}}\right).
\eeq
Hence \eqref{4.11}, $3\sigma''\le 1$, and the definition of $w_{k}^\eps$ yield with some $C>0$,
\beq\label{6.14}
\bbP\left[{T_{u_\eps}}(x_0,\cdot)-t_k(\cdot)\geq \eps^{2{\sigma''}}+ {C} \eps^{3{\sigma''}}\right]\leq 2\exp\left(-\eps^{-2{\sigma'}}\right).
\eeq

In order to upgrade this to \eqref{6612},  let $G_{\eps,T}^k\subseteq \Theta^{k+1}_{\eps,T}\backslash {\Theta^{k}_{\eps,T}}$ be a set containing one point from each cube in $\bbR^d$ with side length $\eps d^{-1/2}$ and all vertices in $\eps d^{-1/2}\bbZ^d$ that has a non-empty intersection with $\Theta^{k+1}_{\eps,T}\backslash {\Theta^{k}_{\eps,T}}$  (recall that $d\le 3$ is  the spatial  dimension). Then clearly $\Theta^{k+1}_{\eps,T}\backslash {\Theta^{k}_{\eps,T}}\subseteq B_\eps(G^k_{\eps,T})$.
If
$
x_0\in G^k_{\eps,T}$,
applying Lemma \ref{L.2.2} to $u^\eps=u_\eps(\eps\cdot,\eps\cdot,\omega)$ yields
\beq\label{6.14'}
T_{u_\eps}(x_0,\omega)\geq \sup_{x\in B_{\eps}(x_0)}T_{u_\eps}(x,\omega)-(2c_0^{-1}+\kappa_0)\eps.
\eeq
This, \eqref{6.14},  and the fact that $|G^k_{\eps,T}|\leq C_AT^{d-1}r^2_\eps\eps^{-d}$ for some $C_A>0$ 
yield with some $C>0$,
\begin{align} \lb{6612}
 \bbP\left[ t_{k+1}(\cdot)-t_k(\cdot)\geq \eps^{2{\sigma''}}+C\eps^{3\sigma''} \right] \notag
&=   \bbP\left[ \sup_{ x\in\Theta^{k+1}_{\eps,T}\backslash \Theta^{k}_{\eps,T}}{T_{u_\eps}}(x,\cdot) - t_k(\omega)\ge \eps^{2{\sigma''}}+C\eps^{3\sigma''} \right] \\
  &\leq 2C_A{T}^{d-1}\eps^{2{\sigma''}-d}\exp\left(-\eps^{-2{\sigma'}}\right).
\end{align}

Next recall that $K= \lceil {T}\eps^{-2\sigma''} \rceil$, and $ T_0\geq\max\{T,1\}$.
Then for $C':=1+2C+2c_0^{-1}+\kappa_0$, with $C$ from \eqref{6612}, we have
\[
K (\eps^{2{\sigma''}}+C\eps^{3\sigma''})+2c_0^{-1}\eps^{{\sigma''}}+\kappa_0\eps\leq 
T+C' T_0 \eps^{{\sigma''}}.
\]
This, \eqref{6612}, and \eqref{6.15}
imply that
\[
\begin{aligned}
    \bbP\left[t_K(\cdot)\geq {T}+C' {T_0} \eps^{{\sigma''}}\right]
    \leq \sum_{k=0}^{K-1}\bbP\left[t_{k+1}(\cdot)-t_k(\cdot)\geq \eps^{2{\sigma''}}+C{\eps}^{3\sigma''}\right]\leq 4{C_A}T_0^{d}\eps^{-d}\exp\left(-\eps^{-2{\sigma'}}\right). 
\end{aligned}
\]
Now \eqref{6.2} and Lemma \ref{L.7.3}(ii) show that
\[
B_{c_1\varepsilon^{\sigma''}/c_0} \left(\Theta^{A,c^*}(T-c_1c_0^{-2}\varepsilon^{\sigma''}) \right) \subseteq \Theta^{A,c^*}(T)\subseteq B_{c_1\varepsilon^{\sigma''}/c_0}\left(\Theta^{A',c'}(T)\right).
\]
Then convexity of $A$ implies $\Theta^{A,c^*}({T}-c_1c_0^{-2}\eps^{{\sigma''}})\subseteq \Theta^{A',c'}({T})$ (note that both these sets are also convex),  so
the definition of $t_K(\omega)$ yields 
\begin{align*}
\bbP \left[ \Theta^{A,c^*}({T}-c_1c_0^{-2}\eps^{{\sigma''}})\not\subseteq  \Gamma_{u_\eps,1-\theta^*}({T}+C' {T_0} \eps^{{\sigma''}},\cdot) \right] 
& \leq \bbP \left[\Theta^{A',c'}({T})\not\subseteq  \Gamma_{u_\eps,1-\theta^*}({T}+C' {T_0} \eps^{{\sigma''}},\cdot) \right]\\
&\leq \bbP\left[t_K(\cdot)\geq {T}+C' {T_0} \eps^{{\sigma''}}\right].
\end{align*}
Therefore
\beq\lb{6.87}
\bbP \left[\Theta^{A,c^*}({T}-c_1c_0^{-2}\eps^{{\sigma''}})\not\subseteq  \Gamma_{u_\eps,1-\theta^*}({T}+C' {T_0} \eps^{{\sigma''}},\cdot) \right]
\leq 4{C_A}T_0^{d}\eps^{-d}\exp\left(-\eps^{-2{\sigma'}}\right).
\eeq

Now let 
\[
T_\eps:=\eps^{{\sigma'}}+C' {T_0} \eps^{{\sigma''}} \qquad\text{and}\qquad C'':=c_1C'+c_1^2c_0^{-2}.
\]
If $t\in [T_\eps,T_0]$, it follows from \eqref{6.87} with $T:=t-C' T_0 \eps^{{\sigma''}} $, and from $(\Theta^{A,c^*}(t))_{c_1s}^0\subseteq \Theta^{A,c^*}(t-s)$ for any $s\in[0,t]$,
that (recall also $T_0\ge 1$, so $C'' T_0 \ge c_1( C' T_0+c_1 c_0^{-2})$)
\beq\lb{6.88}
\bbP\left[ \left(\Theta^{A,c^*}(t)\right)^0_{C'' T_0 \eps^{{\sigma''}}}\not\subseteq  \Gamma_{u_\eps,1-\theta^*}(t,\cdot) \right]\leq 4{C_A}T_0^{d}\eps^{-d}\exp\left(-\eps^{-2{\sigma'}}\right).
\eeq
On the other hand, if $t\in [0, T_\eps]$, then from $c^*\leq c_1$ and $(u_\eps)_t\geq 0$ we obtain 
\beq\lb{6.89}
(\Theta^{A,c^*}(t))^0_{c_1T_\eps}\subseteq A\subseteq \Gamma_{u_\eps,1-\theta^*}(t,\omega).
\eeq

The last two estimates will now yield \eqref{6.1'}.
For any $t\geq s\geq 0$ we clearly have $\Theta^{A,c^*}(s)\subseteq \Theta^{A,c^*}(t)$, and also
$
\Gamma_{u_\eps,1-\theta^*}(s,\cdot)\subseteq \Gamma_{u_\eps,1-\theta^*}(t,\cdot)
$ because $(u_\eps)_t\geq 0$. Then 
\eqref{6.89}   and
\[
\tilde{C}  T_0 \eps^{{\sigma'}}\geq \max\left\{c_1 T_\eps,C'' T_0 \eps^{{\sigma''}}+c_1\eps^{{\sigma'}}\right\},
\]
with $\tilde{C}:=C''+c_1$,
show that
\begin{align*}
   \bbP & \left[\left(\Theta^{A,c^*}(t)\right)^0_{\tilde{C} T_0 \eps^{{\sigma'}}}\not\subseteq  \Gamma_{u_\eps,1-\theta^*}(t,\cdot)\text{ for some }t\in[0,T_0]\right] \\
   &\qquad \le \bbP\left[\left(\Theta^{A,c^*}(t)\right)^0_{C'' T_0 \eps^{{\sigma''}}+c_1\eps^{{\sigma'}}}\not\subseteq  \Gamma_{u_\eps,1-\theta^*}(t,\cdot)\text{ for some }t\in[T_\eps,T_0]\right] \\
&\qquad   \leq  \sum_{  j= \lceil T_\eps\eps^{-\sigma'}\rceil  -1}^{\lceil T_0\eps^{-\sigma'}\rceil -1} \bbP\left[\left(\Theta^{A,c^*}((j+1)\eps^{\sigma'})\right)^0_{C'' T_0 \eps^{{\sigma''}}+c_1\eps^{{\sigma'}}}\not\subseteq  \Gamma_{u_\eps,1-\theta^*}(j\eps^{\sigma'},\cdot) \right].
\end{align*}
Again using
 $(\Theta^{A,c^*}(t))^0_{c_1s}\subseteq \Theta^{A,c^*}(t-s)$  for  $t\geq s\geq 0$, and then 
 \eqref{6.88}, we can continue this estimate via 
\begin{align*}
&  \leq \sum_{  j= \lceil T_\eps\eps^{-\sigma'}\rceil -1 }^{\lceil T_0\eps^{-\sigma'}\rceil -1} 
\bbP\left[\left(\Theta^{A,c^*}(j\eps^{\sigma'})\right)^0_{C'' T_0 \eps^{{\sigma''}}}\not\subseteq  \Gamma_{u_\eps,1-\theta^*}(j\eps^{\sigma'},\cdot) \right]\\
& \leq   4 C_A T_0^{d+1} \eps^{-d-{\sigma'}}\exp\left(-\eps^{-2{\sigma'}}\right).
\end{align*}
Recalling that we wrote $u_\eps$ and $A$ in place of $\tilde{u}_\eps$ and $\eps A^\eps$, this yields \eqref{6.1'} after we let $\eps_0>0$ be small enough (it will then depend on the constants mentioned after \eqref{eps} as well as $A$ and $T_0$).  The proof is thus finished.

\medskip


\medskip







\begin{thebibliography}{10}


\bibitem{AlfGil}
M.~Alfaro and T.~Giletti, 
{\it Asymptotic analysis of a monostable equation in periodic media,}
 Tamkang J. Math. {\bf 47} (2016), no. 1,  1--26. 
 
 
\bibitem{armstrong2015}
{ S.~Armstrong and P.~Cardaliaguet}, 
{\it Stochastic homogenization of
  quasilinear Hamilton-Jacobi equations and geometric motions}, J. Eur. Math. Soc. (JEMS) {\bf 20} (2018), no. 4,  797--864.


\bibitem{aronson1975}
D.~J. Aronson and H.~F. Weinberger, {\it Nonlinear diffusion in population genetics, combustion, and nerve pulse propagation},
Partial Differential Equations and Related Topics, Lecture Notes in Mathematics
{\bf 446}, 5--49, Springer Verlag, 1975.



\bibitem{aronson1978}
D.~J. Aronson and H.~F. Weinberger,  
{\it Multidimensional nonlinear diffusion arising in population genetics},
{Advances in Math.}
  {\bf 30} ({1978}),
  no. {1},
{33--76}.


   

\bibitem{Berrev} 
H.~Berestycki, 
\it The influence of advection on the propagation of fronts in reaction-diffusion equations, 
\rm Nonlinear PDEs in Condensed Matter and Reactive Flows, NATO Science Series C, 569, H. Berestycki and Y. Pomeau eds, Kluwer, Doordrecht, 2003.






 
\bibitem{delfour1994shape}
M.~C. Delfour and J.-P. Zol{\'e}sio,
{\it Shape analysis via oriented distance functions},
J. Funct. Anal.
{\bf 123} (1994),
  no. 1, {129--201}.
  
\bibitem{Fisher}
R.~Fisher, 
{\it The wave of advance of advantageous genes,}
Ann. Eugenics {\bf 7} (1937), no. 4,  355--369.




\bibitem{KPP} 
A. N.~Kolmogorov, I. G.~Petrovskii, and N.S.~Piskunov, 
\it \'Etude de l'\'equation de la diffusion avec croissance de la quantit\'e de mati\`ere et son application \`a un probl\`eme biologique, 
\rm Bull. Moskov. Gos. Univ. Mat. Mekh. {\bf 1} (1937), 1--25.

\bibitem{kosy}
E. Kosygina, F. Rezakhanlou and S. R. S. Varadhan,
{\it Stochastic homogenization of Hamilton-Jacobi-Bellman equations},
 Comm. Pure Appl. Math. {\bf 59} (2006) no. 10, 1489--1521.



\bibitem{LinZla}
J.~Lin and A.~Zlato{\v{s}}, 
{\it Stochastic homogenization for
reaction--diffusion equations}, 
Arch.~Ration.~Mech.~Anal.
  {\bf 232} (2019), no. 2,   813--871.
  
  
\bibitem{LPV}
P.-L.~Lions, G.~Papanicolaou, and S.~R.~S. Varadhan.
{\it Homogeneization of Hamilton--Jacobi equations}, unpublished preprint, 1986.

\bibitem{LioSouPer}
P.-L.~Lions and P. E.~Souganidis,
{\it Homogenization of degenerate second-order PDE in periodic and almost periodic environments and applications},
 Ann. Inst. H. Poincar{\'e} Anal. Non Lin{\'e}aire
 {\bf 22}  (2005), no. 5, 667--677.
 
\bibitem{LioSou} P.-L.~Lions and P. E.~Souganidis, 
{\it Homogenization of ``viscous'' Hamilton-Jacobi equations in stationary ergodic media},
\rm  Comm.~Partial Differential Equations  {\bf 30}  (2005), no. 3, 335--375.

 



\bibitem{MajSou}
{ A.~J. Majda and P.~E. Souganidis}, {\it Large scale front dynamics for
  turbulent reaction-diffusion equations with separated velocity scales},
  Nonlinearity {\bf 7} (1994), no. 1,   1--30.
  



\bibitem{RezTar}
F.  Rezakhanlou and J. E.  Tarver,
{\it Homogenization for stochastic Hamilton-Jacobi equations},
Arch.~Ration.~Mech.~Anal. {\bf 151} (2000), no. 4,   277--309. 


\bibitem{47souganidis}
{ P.~E. Souganidis}, {\it Front propagation: theory and applications}, in
  Viscosity solutions and applications, Springer, Berlin, 1997,  186--242.
  
  
\bibitem{sou1999}
P. E.~Souganidis,
{\it Stochastic homogenization of Hamilton--Jacobi equations and some applications},
  Asymptot. Anal.  {\bf 20} (1999), no. 1, 1--11.
  


   
   
\bibitem{xin2000}
{ J.~Xin}, {\it Front propagation in heterogeneous media},
SIAM Rev. {\bf 42} (2000), no. 2, 161--230.

   
\bibitem{ZhaZla}
Y.~P. Zhang and A.~Zlato{\v{s}},
{\it Long time dynamics for combustion in random media}, preprint, 
arXiv:2008.02391.
   

  

\bibitem{ZlaBist}
A.~Zlato{\v{s}}, 
{\it Existence and non-existence of transition fronts
  for bistable and ignition reactions}, Ann.~Inst.~H.~
  Poincar{\'e} Anal.~Non Lin{\'e}aire {\bf 34} (2017), no. 7,
   1687--1705.

\bibitem{ZlaInhomog}
A.~Zlato{\v{s}}, 
{\it Propagation of  reactions in inhomogeneous media}, Comm.~Pure Appl.~Math. {\bf 70} (2017), no. 5, 884--949.
  
  


\end{thebibliography}

\end{document}